%% file: Vikram-Interpolating-Sequence.tex
\documentclass[12pt]{article}

\usepackage{mathrsfs}
\usepackage{amssymb,amsthm,amsmath,amssymb,wrapfig,dsfont,authblk}
\usepackage{mathtools}
\usepackage{blkarray}
\usepackage{mathrsfs}
\usepackage{mathtools}
\usepackage{amsmath}
\usepackage{amsfonts}
\usepackage{amsmath,amsfonts,amssymb,amsthm,epsfig,epstopdf,titling,url,array}
\usepackage{lipsum}
\usepackage{mathtools, booktabs, bigstrut}
\usepackage{blindtext}
\usepackage{graphicx}

\newcommand{\D}{\mathbb{D}}
\newcommand{\lam}{\lambda}

\newcommand{\bcdot}{\boldsymbol{\cdot}}
\newcommand{\nat}{\mathbb{N}}
\newcommand{\lrarw}{\longrightarrow}
\newcommand{\Cbpsi}{\mathcal{C}_b (\Psi)}

\newcommand\intf[4]{\genfrac{#1}{#2}{0.5pt}{0}{#3}{#4}}

\input{includes}

\newcommand{\A}{\mathscr{A}}
\newcommand{\B}{\mathscr{B}}

\newcommand{\X}{\mathscr{X}}
\newcommand{\E}{\mathcal{E}}

\newcommand{\U}{\mathscr{U}}

\newcommand{\Y}{\mathscr{Y}}
\newcommand{\I}{\mathbb{I}}

\newcommand{\K}{\mathcal{K}}
\newcommand{\SA}{\mathscr{S\mspace{-5mu}A}}
\newcommand{\HH}{\mathscr{H}}

\newcommand{\GammaDomainRange}{\Gamma: \Omega \times \Omega\rightarrow B(\mathcal{C}_b(\Psi), B(\Y))}

\providecommand{\keywords}[1]{\textbf{\textit{Keywords---}} #1}

\begin{document}
	
	\title{Interpolating sequences for the Banach algebras generated by a class of test functions}
	\author{Anindya Biswas \thanks{Department of Mathematics, Indian Institute of Science, Bangalore 560012, India. e-mail: anindyab@iisc.ac.in} 
	\and Vikramjeet Singh Chandel \thanks{Harish-Chandra Research Institute, Prayagraj (Allahabad), 211019, India. e-mail: vikramjeetchandel@hri.
	  res.in}}

	\date{}
	\maketitle

	\begin{abstract}
	Given a domain $\Omega$ in $\mathbb{C}^n$ and a collection of test functions $\Psi$ on
	$\Omega$, we consider the complex-valued $\Psi$-Schur-Agler class associated to the pair
	$(\Omega,\,\Psi)$. In this
	article, we characterize interpolating sequences for the associated Banach algebra of which the
	$\Psi$-Schur-Agler class is the closed unit ball. When $\Omega$ is the unit disc $\D$ in the 
	complex plane $\mathbb{C}$ and the class of test function includes only the identity function on $\D$,
	the aforementioned algebra is the algebra of bounded holomorphic functions on $\D$ and in this case,
	our characterization reduces to the well known result by Carleson. Furthermore, we present several other
	cases of the pair $(\Omega,\,\Psi)$, where our main result could be applied to characterize interpolating 
	sequences which also show the efficacy of our main result. 
		\noindent

\end{abstract}


\section{Introduction.}
\subsection{Interpolating sequences: An overview}
Let $\Omega$ be a bounded domain in $\mathbb{C}^n$. Let $\A$ be a Banach algebra of bounded functions on $\Omega$ with the norm 
$||\bcdot||_{\A}$ and having the property 
that $||f||_{\A}\geq ||f||_{\infty}:=\sup_{z\in\Omega}|f(z)|$ for every $f\in \A$. In what follows, $l^{\infty}(\nat)$ will denote the Banach algebra of 
all bounded sequences with sup-norm. Given a sequence $\{z_i:i\in \mathbb{N}\}\subset\Omega$, we consider the following linear map:
\begin{equation}\label{E:eval-L}
L:\big(\A,\,||\bcdot||_{A}\big)\lrarw l^{\infty}(\nat), \ \ \ \text{defined by $L(\phi):=\{\phi(z_i)\}$ for all $\phi\in\A$}.
\end{equation}
Observe that $\sup_{i\in\nat}|\phi(z_i)|\leq ||\phi||_{\infty}\leq ||\phi||_{\A}$, hence $L$ above is a bounded linear operator on $\A$. 
Consider the following abstract interpolation problem:

\begin{itemize}
\item[(IS)] Given a Banach algebra $(\A,\,||\bcdot||_{\A})$ of bounded functions on $\Omega$ with the property that $||\phi||_{\A}\geq ||\phi||_{\infty}$ for every $\phi\in \A$. Characterize all those sequences $\{z_i:i\in \mathbb{N}\}$ in $\Omega$ for which the bounded linear map $L$, as defined above, is a surjective map. 
\end{itemize}

\noindent{A sequence $\{z_i:i\in \mathbb{N}\}\subset\Omega$ for which the map $L$ is surjective will be called an \textit{interpolating sequence for the algebra $\A$}.
\smallskip

Note that if $\{z_i:i\in \mathbb{N}\}$ is an interpolating sequence for $\A$ then\,---\,from the open mapping theorem\,---\,there exists a $\delta>0$ such that $l^{\infty}_1(\nat)$:=the closed unit ball of $l^{\infty}(\nat)$, is contained in $L(\delta\,\A_1)$, where $\A_1$ denotes the closed unit ball of $\A$ in its norm. The smallest of such a $\delta$ is called
\textit{the constant of interpolation} for the interpolating sequence $\{z_i:i\in \mathbb{N}\}\subset\Omega$. 
This, in particular, implies that for each $i\in\nat$, there exists $\phi_i\in\A$ with $||\phi_i||_{\A}\leq M$, such that $\phi_i(z_j)=\delta_{i,j}$, where $\delta_{i,j}$ is the Kronecker symbol and $M$ is the constant of
interpolation associated to $\{z_i:i\in \mathbb{N}\}$. A sequences $\{z_i:i\in \mathbb{N}\}$ for which there exists a sequence $\{\phi_i\in\A:i\in \mathbb{N}\}$ that is uniformly bounded and has 
the property that $\phi_i(z_j)=\delta_{i,j}$ are called \textit{strongly separated sequences}. So every interpolating sequence $\{z_i:i\in \mathbb{N}\}$ for $\A$ is strongly separated. Given a sequence $\{z_i:i\in \mathbb{N}\}$ we shall call it is {\em weakly separated by $\A$} if there exists $R>0$ such that for each pair
$i\neq j$ there exists $\phi_{i,j}\in\A$ with $||\phi_{i,j}||_{\A}\leq R$ such that $\phi_{i,j}(z_i)=1$ and $\phi_{i,j}(z_j)=0$.
\smallskip 

The origins of the problem (IS) lies in the case when $\Omega=\D$, where $\D$ denotes the open unit disc in the
complex plane centered at $0$, and with $\A=H^{\infty}(\D)$:= the bounded holomorphic functions in the unit disc with the sup-norm.
Carleson in $1958$ proved the following theorem.
\smallskip

\noindent{\bf Result~1. (Carleson, \cite{LC1958})} Let $\{\lam_i:i\in \mathbb{N}\}\subset\D$ be a sequence in $\D$. Then the following are equivalent.
\begin{enumerate}
\item $\{\lam_i:i\in \mathbb{N}\}$ is an interpolating sequence for $H^{\infty}(\D)$. 
\item $\{\lam_i:i\in \mathbb{N}\}$ is weakly separated and the atomic measure $\sum_{i\in\nat}\,\big(1-|\lam_i|^2\big)\,\delta_i$ is a Carleson measure for the Hardy 
space $H^2(\D)$.
\item $\{\lam_i:i\in \mathbb{N}\}$ is strongly separated.
\end{enumerate}
The reader is referred to \cite[Chapter~9]{A-M} for the definition of Carleson measure. What is essential here is that the Carleson measure condition can be equivalently stated in terms of the boundedness of the Grammian operator on $l^2$ corresponding to the Szego kernel on $\mathbb{D}$. For this purpose
let us introduce the Grammian associated 
with a positive kernel $k$ on $\Omega$. See Section~\ref{S:prelims} for the definition of a positive kernel $k$ and the reproducing kernel Hilbert space $\mathcal{H}_k$ that is associated to it.
Given a positive kernel $k$ on $\Omega$ and a sequence $\{z_i:i\in \mathbb{N}\}\subset\Omega$, let us denote by 
$k_i$ the kernel function at $z_i$, i.e., $k(\bcdot,z_i)$ and write 
$k_{i,j}:=\langle k_j,\,k_i\rangle$. Let $g_i:=k_i/||k_i||$ be the normalized kernel functions. The \textit{Grammian associated to the sequence $\{z_i:i\in \mathbb{N}\}$}
is the infinite matrix $G$ given by 
\[
G_{i,j}:=\langle g_j,\,g_i\rangle=\intf{}{}{k_{i,j}}{||k_i||\,||k_j||}.
\]

It is a fact that the Grammian associated to a sequence $\{z_i:i\in \mathbb{N}\}$ is bounded on $l^2$ if and only if the measure
$\sum_{i=1}^n ||k_i||^{-2}\,\delta_i$ is a Carleson measure for $\mathcal{H}_k$; see 
\cite[Proposition~9.5]{A-M}. The Hardy space on the unit disk, $H^2(\D)$, is a reproducing kernel
Hilbert space with the kernel $k$ being the Szeg\H{o} kernel. So the Carleson measure condition in the above result is equivalent to boundedness
of the Grammian matrix associated to the Szeg\H{o} kernel and the sequence $\{\lam_i:i\in \mathbb{N}\}$ above.
\smallskip

Shapiro--Shield in \cite{S-S} gave an alternative proof of Carleson's result by replacing the notions of Carleson measure condition and strong sepration
by conditions on the the Grammian matrix associated to a sequence. They also 
considered interpolating sequences for many other holomorphic function spaces on the unit disc. 
\smallskip

An important case of the problem (IS) is when $\A$ is the multiplier algebra, denoted by $Mult(\mathcal{H}_k)$,
of a reproducing kernel Hilbert space 
$\mathcal{H}_k$ associated to a kernel $k$, together with the multiplier norm. This was initiated by Marshall--Sundberg \cite{MS94} and by C.\,Bishop \cite{Bis94}. They also introduced a notion of interpolating sequences for the Hilbert space $\mathcal{H}_k$ and observed that the set of interpolating sequences for multiplier algebras
is contained in the set of interpolating sequences for the Hilbert space $\mathcal{H}_k$. Moreover, if the kernel
satisfies the Pick property, then the two notions of interpolating sequences coincide;
see \cite[Theorem~9.19]{A-M}. 
This is important since interpolating sequences for separable Hilbert spaces are exactly
those for which the Grammian is bounded and bounded from below.
\smallskip

A well studied class of positive 
kernels is the family of complete Nevanlinna--Pick kernels (see \cite{A-M} for the definition) that satisfy a more
stronger form of Pick property. 
A characterization of the interpolating sequences 
for the multiplier algebra for this class of kernels is now completely known. 
\smallskip

\noindent{\bf Result~2.} Let $\{\lam_i:i\in \mathbb{N}\}\subset\Omega$ be a sequence and let $k$ be an irreducible complete Nevanlinna--Pick kernel. 
Then the following are equivalent:
\begin{itemize}
\item[(IM)] the sequence is interpolating for $Mult(\mathcal{H}_k)$,
\item[(IH)] the sequence is interpolating for $\mathcal{H}_k$,
\item[(S+C)] the sequence is weakly separated and the Grammian associated to the sequences\\ is bounded.
\end{itemize}
As mentioned before the equivalence of (IM) and (IH) above was established by
Marshall-Sundberg and by Bishop. The implication (IH)$\implies$(S+C) holds, in general, for any reproducing kernel Hilbert space; see e.g. \cite{A-M} or \cite{KS}. The implication (S+C)$\implies$(IM) for irreducible complete Nevanlinna-Pick kernels is established in a recent article by
Aleman-Hartz--McCarthy--Richter \cite{A-H-M-R}, where they applied Marcus--Spielman--Srivastava theorem, a path-breaking result that established the Kadison-Singer conjecture. 
\smallskip

The condition (S+C) implies strong separation with respect to $Mult(\mathcal{H}_k)$ for kernels
having Pick property; see \cite[Theorem~9.43]{A-M}. On the other hand, both Marshall–Sundberg \cite{MS94} and Bishop \cite{Bis94} have shown that strong separation does not imply (IM) in the case of the Dirichlet space of the unit disc.
\smallskip

\subsection{Test functions}\label{test}

In this article, we shall address the problem (IS) with $\A$ being those Banach algebras of bounded functions that are obtained by taking intersections of multiplier algebras of certain reproducing kernel Hilbert spaces associated with a class of test functions.
These are algebras which are not necessarily multiplier algebras
of a reproducing kernel Hilbert spaces and were first introduced by Jim Agler.
\smallskip

Let $\Omega$ be a bounded domain in $\mathbb{C}^N$ and let $\Psi$ be a family of functions on $\Omega$. 
We say $\Psi$ is a collection of \textit{ test functions} on $\Omega$ if the following conditions hold:
\begin{enumerate}
	\item sup $\{|\psi(x)|:\psi\in \Psi\}<1$ for each $x\in \Omega$.
	\item For each finite subset $F$ of $\Omega$, the collection $\{\psi |_F:\psi\in \Psi\}$ together with unity generates the algebra of all
	 $\mathbb{C}$-valued functions on $F$.
\end{enumerate}
The second condition is not essential part of the definition, but it makes some situations simpler (see \cite{B_H} and \cite{D-M}).
The collection $\Psi$ is a natural topological subspace of $\overline{\mathbb{D}}^\Omega$ equipped with the product topology. For every
$x\in \Omega$, there is an element $E(x)$ in $\mathcal{C}_b(\Psi)$, {\em the $C^*$-algebra of all bounded functions on} $\Psi$, such that
$E(x)(\psi)=\psi(x)$. Clearly,
\begin{math}
\| E(x)\|= \text{sup}_{\psi\in \Psi} |\psi(x)|<1
\end{math}
for each $x\in \Omega$. The functions $E(x)$ will be used at several places in this paper.
\smallskip

Given $\Omega$ and a collection of test functions $\Psi$, let us denote by $\mathcal{K}_{\Psi}(\mathbb{C})$ the set of all 
$\mathbb{C}$-valued positive kernels $k$ on $\Omega$ for which 
the operator $M_{\psi}:\mathcal{H}_k\lrarw\mathcal{H}_k$ defined by $M_{\psi}(f):=f\,\psi$, for all $f\in\mathcal{H}_k$, is a 
contraction for each $\psi\in\Psi$. Recall a contraction on a Hilbert space is a bounded linear operator whose operator norm is atmost $1$.

\begin{definition}
Let us denote by $H^\infty_\Psi(\mathbb{C})$ the collection of such $\mathbb{C}$-valued functions $\phi:\Omega\lrarw\mathbb{C}$ for which 
there exists a $C>0$ having the following property:
\begin{itemize}
\item[$(*)$] for each $k\in \mathcal{K}_{\Psi}(\mathbb{C})$, the bounded linear operator $M_{\phi}:\mathcal{H}_k\lrarw\mathcal{H}_k$ defined by 
 $M_{\phi}(f):=f\,\phi$ for all $f\in\mathcal{H}_k$, is a bounded linear operator with $||M_{\phi}||_{\mathcal{H}_k}\leq C$.
 \end{itemize}
 Given $\phi\in H^\infty_\Psi(\mathbb{C})$, define:
 \begin{equation}\label{E:Psinorm}
 ||\phi||_{\Psi}:=\inf\big\{C\,:\,\text{$C$ satisfying the property $(*)$ above}\big\}.
 \end{equation}
 \end{definition}
\noindent It turns out that $H^\infty_\Psi(\mathbb{C})$ is a Banach algebra with norm $||\bcdot||_{\Psi}$. The 
 \textit{scalar-valued $\Psi$-Schur-Agler class}, denoted by $\SA_{\Psi}(\mathbb{C})$, is defined by 
 $\SA_{\Psi}(\mathbb{C}):=\big\{\phi\in H^\infty_\Psi(\mathbb{C})\,:\,||\phi||_{\Psi}\leq 1\big\}$. It is a general fact that if $\phi\in H^\infty_\Psi(\mathbb{C})$ then $||\phi||_{\infty}:=\sup\{|\phi(z)|\,:\,z\in\Omega\}\leq ||M_{\phi}||_{\mathcal{H}_k}$ for all $k\in K_{\Psi}$. It follows from this that $||\phi||_{\infty}\leq ||\phi||_{\Psi}$.
We are now ready to state the main result of this article concerning the
problem (IS) in the case when $\A=H^{\infty}_{\Psi}(\mathbb{C})$. 

\begin{theorem}\label{T:maint}
Let $\Omega$ be a bounded domain and let $\Psi$ be a family of test functions. Consider the Banach algebra
$H^{\infty}_{\Psi}(\mathbb{C})$ consisting of bounded functions with the norm $||\bcdot||_{\Psi}$ as above. 
Let $\{w_j:j\in\nat\}\subset\Omega$ be a sequence in $\Omega$. Then the following are equivalent.
\begin{enumerate}
\item $\{w_j:j\in\nat\}$ is an interpolating sequence for $H^\infty_\Psi(\mathbb{C})$. 
\smallskip

\item For all admissible kernels $k\in \mathcal{K}_{\Psi}$, the normalized Grammians $G_k$ are uniformly bounded from below, i.e., there exists $N>0$ such that
$G_k\geq (1/N)\,\mathbb{I}$ for all $k\in \mathcal{K}_{\Psi}$. 
\smallskip

\item $\{w_j:j\in\nat\}$ is strongly separated and for all kernels $k\in \mathcal{K}_{\Psi}$, the normalized Grammians $G_k$ are uniformly bounded, i.e., there 
exists $M>0$ such that $G_k\leq M\,\mathbb{I}$ for all $k\in \mathcal{K}_{\Psi}$.

\item $(2)$ and $(3)$ above holds together. 

\end{enumerate}
\end{theorem}
\noindent{Here, and elsewhere in the article $\mathbb{I}$ shall denote the identity operator.}
\smallskip

\noindent\textbf{Remark.} Let $k,\,l\in \mathcal{K}_{\Psi}$ be such that $l\,=\,g\,k$ for some positive kernel $g$. Then if $c\,\I-G_k\geq 0$ for some $c>0$ then $c\,\I-G_l\geq 0$. To see that, note that $G_g\geq 0$ and by the Schur-product theorem we get $(c\,\I-G_k)\,G_g\geq 0$. Notice now the Schur product of $\I$ and $G_g$ is $\I$ and $G_l\,=\,G_k\,G_g$ whence the conclusion. Proceeding similarly if $G_k-d\,\I\geq 0$ for some $d>0$, then $G_l-d\,\I\geq 0$, where $l,\,k\in\mathcal{K}_{\Psi}$ are as before. 
\smallskip

When $\Omega=\mathbb{D}$ and $\Psi=\{z\}$ then every $k\in\mathcal{K}_{\Psi}$ is of the form
$k\,=\,s\,g$, where $s$ denotes the Szeg\H{o} kernel and $g$ is some positive kernel. It follows from the remark
above that conditions $(2)$ and $(3)$ in Theorem~\ref{T:maint} has to be satisfied only for the Szeg\H{o} kernel.
It is a fact that $H^{\infty}_{\Psi}(\mathbb{C})$ is $H^{\infty}(\D)$ in this case. This leads 
to a characterization of interpolating sequences for $H^{\infty}(\D)$ which is equivalent to Result~1
(\cite[Section~9.5]{A-M}) by Carleson.

We shall present the proof of Theorem~\ref{T:maint} in Section~\ref{S:proofmaint}. Later, in Section~\ref{examples}, we shall provide several other examples where the above theorem 
could be applied in characterizing interpolating sequences for those algebras that could be realized as 
$H^{\infty}_{\Psi}(\mathbb{C})$ for some appropriately chosen $\Psi$.

\section{Preliminaries}\label{S:prelims}
In this section, we prove a few important lemmas and gather certain basic tools that will be needed in upcoming sections. We first begin with the definition of $\Psi$-Schur-Agler class functions in operator-valued setting. To  do that, we need to recall various notions of positive kernel on a domain $\Omega$.
\smallskip 

A positive kernel $k$ on a set $\Omega$ is a function $k:\Omega\times \Omega \rightarrow \mathbb{C}$ such that for any $n\geq 1$, any $n$
points $x_1,\ldots,x_n$ in $\Omega$ and any $n$ complex numbers $c_1,\ldots, c_n$, we have
\begin{align*}
\sum_{i=1}^{n}\sum_{j=1}^{n}\overline{c_i}c_j k(x_i,x_j)\geq 0.
\end{align*}
 
Let $\mathcal{E}$ is a Hilbert space and $k:\Omega\times \Omega\rightarrow B(\mathcal{E})$ is a function, then $k$ is called a positive 
kernel if for any $n\geq 1$, any $n$ points $x_1,\ldots,x_n$ in $\Omega$ and any $n$ vectors $\mathbf{e_1,\ldots, e_n}$ in $\mathcal{E}$, we have
\begin{align}\label{op. valued kernel}
\sum_{i=1}^{n}\sum_{j=1}^{n}\langle k(x_i,x_j)\mathbf{e_j}, \mathbf{e_i} \rangle\geq 0.
\end{align}
We also recall the notion of completely positive kernels here. Let $\A$ and $\B$ be two $C^*$-algebras and let $\Gamma$ be a function
on $\Omega \times \Omega$ taking values in $B(\A, \B)$ (space of all bounded linear operators from $\A$ to $\B$). $\Gamma$ is called a 
\textit{completely positive kernel} if 
\begin{align} \label{compositivekernel}
	\sum_{i,j =1}^{n} b_i^* \  \Gamma(x_i, x_j)(a_i ^*a_j) b_j \geq 0
\end{align}
for all $n \geq 1$, $a_1, a_2, \dots, a_n \in \A$, $b_1, b_2, \dots, b_n \in \B$ and $x_1, x_2, \dots, x_n \in \Omega$.\\

\subsection{The $\Psi$-Schur-Agler Class: General Case}\label{SAClass}
Given a Hilbert space $\E$ and a $B(\E)$-valued kernel $K$ \textit{(satisfying (\ref{op. valued kernel}))}
on $\Omega$, there is a Hilbert space $\mathcal{H}(K)$ of $\E$-valued functions on $\Omega$ such that span of the set
\begin{align*}
\{K(\cdot, \omega)\mathbf{e}:\mathbf{e}\in \E,\, \omega\in \Omega \}
\end{align*}
is dense in $\mathcal{H}(K)$ and for any $\mathbf{e}\in \E$, $\omega \in \Omega$ and $h\in \mathcal{H}(K)$, we have
\begin{align*}
\langle h, K(\cdot, \omega)\mathbf{e}\rangle_{\mathcal{H}(K)} = \langle h(\omega), \mathbf{e} \rangle_{\E}.
\end{align*}

Given a set of test functions $\Psi$ on $\Omega$, a kernel $K:\Omega \times \Omega\rightarrow B(\E)$ is said to be {\em $\Psi$-admissible} if
the map $M_\psi$, sending each element $h\in \mathcal{H}(K)$ to $\psi\cdot h$, is a contraction on $\mathcal{H}(K)$.
We denote the set of all $B(\E)$-valued $\Psi$-admissible kernels by $\K_\Psi(\E)$. For two Hilbert spaces $\U$ and $\Y$,
we say that {\em $S: \Omega \rightarrow B(\U, \Y)$ is in $H^\infty_\Psi(\U, \Y)$}
if there is a constant $C$ such that the $B(\Y \otimes \Y)$-valued function
\begin{equation} \label{C} (C^2 I_\Y - S(x)S(y)^*)\otimes k(x,y)  \end{equation}
is a positive $B(\Y\otimes \Y)$-valued kernel for every $k$ in $\K_\Psi(\Y)$. If $S$ is in $H^\infty_\Psi(\U, \Y)$,
then we denote by $C_S$ the smallest $C$ which satisfies \eqref{C}.
The $\Psi$-Schur-Agler class, denoted by $\SA_\Psi(\U, \Y)$, is the set of those
$S\in H^\infty_\Psi(\U, \Y)$ for which $C_S=1$. Also
observe when $\U=\mathbb{C}=\Y$, then  $H^\infty_\Psi(\U, \Y)=H^{\infty}_{\Psi}(\mathbb{C})$.
We begin with the following important lemma.

\begin{lemma}\label{L:prodschuragler}
Let $\X,\Y,\HH$ be Hilbert spaces such that either $\HH=\mathbb{C}$ or $\HH=\Y$. Let $g\in\SA_{\Psi}(\X,\Y)$ and $f\in\SA_{\Psi}(\Y,\HH)$.
Then if we set $fg(z):= f(z)g(z)$ for all $z\in\Omega$ then $fg\in\SA_{\Psi}(\X,\HH)$. 
\end{lemma}

\begin{proof}
Notice that for each $z\in\Omega$, $fg(z)\in B(\X,\,\HH)$. We need to show that for each $k\in\K_\Psi(\HH)$,
$\big(I_{\HH}-f(z)g(z)g(w)^{*}f(w)^{*}\big)\otimes k(z,\,w)$ is a positive $B(\HH\otimes\HH)$-valued kernel.
Note that
\begin{align}
\big(I_{\HH}-(fg)(z)(fg)(w)^{*}\big)\otimes k(z,\,w)=&\big(I_{\HH}-f(z)g(z)g(w)^{*}f(w)^{*}\big)\otimes k(z,\,w)\nonumber\\
=&\big(I_{\HH}-f(z)f(w)^{*}\big)\otimes k(z,w)\nonumber\\
+&f(z)\big(I_{\Y}-g(z)g(w)^{*}\big)f(w)^{*}\otimes k(z,w).
\end{align}
The expression $\big(I_{\HH}-f(z)f(w)^{*}\big)\otimes k(z,w)$ is positive as $f\in\SA_{\Psi}(\Y,\HH)$.
We now consider $f(z)\big(I_{\Y}-g(z)g(w)^{*}\big)f(w)^{*}\otimes k(z,w)$ in two cases:
\smallskip

\noindent{\bf Case 1.} $\HH=\mathbb{C}$. 
\smallskip

\noindent In this case, the above expression becomes $\big(f(z)(I_{\Y}-g(z)g(w)^{*})f(w)^{*}\big)k(z,w)$. Now we make
the following claim.
\smallskip

\noindent{\bf Claim.} For a $\mathbb{C}$-valued $\Psi$-admissible kernel $k$, if we let $K:=k\,I_{\Y}$ then $K$ is a $B(\Y)$-valued kernel that is $\Psi$-admissible.
\smallskip

\noindent To see that  $K:=k\,I_{\Y}$ is a kernel, choose $y_1,\dots,y_n\in\Y$ and $z_1,\dots,z_n\in\Omega$ and compute
\[
\sum_{i,j=1}^n \langle K(z_i,z_j)y_j,y_i\rangle=\sum_{i,j=1}^n\langle k(z_i,z_j)y_j, y_i\rangle\\
=\sum_{i,j=1}^n k(z_i,z_j)\langle y_j, y_i\rangle.
\]
Now fix a basis $\{e_{\alpha}\}$ of $\Y$ and write $y_i=\sum_{\alpha}y_{\alpha,i}\,e_{\alpha}$. If we define $\bar{y}_i=
\sum\bar{y}_{\alpha,i}\,e_{\alpha}$, then the last expression above is equal to 
\[
\sum_{i,j=1}^n k(z_i,z_j)\langle \bar{y}_j, \bar{y}_i\rangle
\]
which is positive. Consequently $K$ above is positive too. We now show that $K$ is $\Psi$-admissible. Choose $\psi\in\Psi$ and $z_1,\dots,z_n\in\Omega$
and $y_1,\dots,y_n\in\Y$  and compute
\begin{align*}
||\sum_{j=1}^n K(\bcdot,z_j)y_j||^2-||M_{\psi}^{*}(\sum_{j=1}^n K(\bcdot,z_j)h_j)||^2
&=\sum_{i,j=1}^n\big(\langle k(z_i,z_j)y_j,y_i\rangle-\psi(z_i)\psi(z_j)\langle k(z_i,z_j)y_j, y_i\rangle\big)\\
&=\sum_{i,j=1}^n\big(1-\psi(z_i)\psi(z_j)\big)k(z_i,z_j)\langle h_j, h_i\rangle.
\end{align*}
The last expression above is nonnegative and hence $||M_{\psi}^*||=||M_{\psi}||\leq 1$. Since the above holds for any $\psi\in\Psi$, the claim is established. 
\smallskip

Coming back to the proof of Case~1 above, using the claim we see that $(I_{\Y}-g(z)g(w)^{*})\otimes (k I_{\Y})$ is a $B(\Y\otimes\Y)$-valued kernel.
Choose a basis element $e$ of $\Y$ and consider $u_i=f(z_i)^{*}\otimes(c_i e)$. Then 
\begin{align*}
0\leq&\sum_{i,j=1}^n\Big{\langle} \Big(\big(I_{\Y}-g(z_i)g(z_j)^{*}\big)\otimes\big(k(z_i,z_j)I_{\Y}\big)\Big)u_j, u_i\Big{\rangle}\nonumber\\
=&\sum_{i,j=1}^n \big{\langle} \big(I_{\Y}-g(z_i)g(z_j)\big)f(z_j)^{*}(1), f(z_i)^{*}(1)\big{\rangle}\,k(z_i,z_j)\,\langle c_je,c_ie\rangle\\
=&\sum_{i,j=1}^n\bar{c}_i c_j\big{\langle}f(z_i)\big(I_{\Y}-g(z_i)g(z_j)\big)f(z_j)^{*}(1), 1\big{\rangle}\,k(z_i,z_j)\\
=&\sum_{i,j=1}^n \bar{c}_i c_j\Big(f(z_i)\big(I_{\Y}-g(z_i)g(z_j)^{*}\big)f(z_j)^{*}\Big)\,k(z_i,z_j)
\end{align*}
Thus $f(z)\big(I_{\Y}-g(z)g(w)^{*}\big)f(w)^{*}\,k(z,w)$ is a positive kernel and we are done in this case.
\smallskip

\noindent{\bf Case 2.} $\HH=\Y$. 
\smallskip

\noindent So $f\in\SA_{\Psi}(\Y,\Y)$ and we need to show that $fg\in\SA_{\Psi}(\X,\Y)$. Consider any $B(\Y)$-valued $\Psi$-admissible 
kernel $K$ and consider $f(z)\big(I_{\Y}-g(z) g(w)^{*}\big) f(w)^{*}\otimes K(z,w)$
which has to be shown a positive $B(\Y\otimes \Y)$-valued kernel. For this purpose take $y_{i1}\otimes y_{i2}\in\Y\otimes\Y$, $1\leq i\leq n$, and compute:
\begin{align*}
&\sum_{i,j=1}^n\Big{\langle}[\Big(f(z_i)\big(I_{\Y}-g(z_i) g(z_j)^{*}\big)\Big)f(z_j)\otimes K(z_i,z_j)](y_{j1}\otimes y_{j2}), y_{i1}\otimes y_{i2}
\Big{\rangle}\\
=&\sum_{i,j=1}^n{\Big{\langle}f(z_i)\big(I_{\Y}-g(z_i)g(z_j)^{*}\big)f(z_j)^{*}y_{j1},\,y_{i1}\Big{\rangle}}{\langle K(z_i,z_j)y_{j2},y_{i2}\rangle}\\
=&\sum_{i,j=1}^n{\Big{\langle}\big(I_{\Y}-g(z_i)g(z_j)^{*}\big)V_j,\,V_i\big{\rangle}}\,{\langle K(z_i,z_j)y_{j2}, y_{i2}\rangle},
\end{align*}
where in the last expression $V_i=f(z_i)^{*}y_{i1}$ for all $i,\,1\leq i\leq n$. The last expression, furthur, can be written as 
\[
\sum_{i,j=1}^n\Big{\langle}\big[\big(I_{\Y}-g(z_i)g(z_j)^{*}\big)\otimes K(z_i,z_j)\big](V_j\otimes y_{j2}),\,(V_i\otimes y_{i2}\Big{\rangle},
\]
which is non-negative owing to the fact that $g\in\SA_{\Psi}(\X,\Y)$ and $K$ is a $B(\Y)$-valued $\Psi$-admissible positive kernel. Thus 
$f(z)\big(I_{\Y}-g(z) g(w)^{*}\big) f(w)^{*}\otimes K(z,w)$ is a positive kernel for any $K$ that is $\Psi$-admissible and we are done in this case too.
\end{proof}

\subsection{$\Psi$-unitary Colligations and Realization of $\Psi$-Schur-Agler class functions}
Let $\X$, $\U$ and $\Y$ be Hilbert spaces and let $\Psi$ be a fixed set of test functions. By a $\Psi$-unitary colligation,
we mean a pair $(U, \rho)$ where $U$ is a unitary operator from $\X \oplus \U$ to $\X \oplus \Y$, and
$ \rho: \mathcal{C}_b(\Psi) \rightarrow B(\X)$ is a $*$-representation. If we write $U$ as
\[
U =
\bordermatrix{ & \X & \U \cr
	\X & A & B \cr
	\Y & C & D}, \qquad
\]
then we can define a bounded $B(\U, \Y)$ valued function on $\Omega$, given by
\begin{align}
  f(x) = D + C \rho (E(x)) (\I_{\X}- A \rho(E(x)))^{-1} B \ \ \forall \ x \in \Omega,
\end{align}
equivalently,
\begin{align}
  f(x) = D + C(\I_{\X}- \rho (E(x)) A)^{-1} \rho (E(x)) B  \ \ \forall \ x \in \Omega.
\end{align}
This $f$ is called the transfer function associated with $(U,\rho)$. Since $U^*$ is also a unitary, we have that
\begin{align*}
g(x) = D^* + B^*(\I_{\X}- \rho (E(x)) A^*)^{-1} \rho(E(x)) C^*
\end{align*}
 is the transfer function of the colligation $(U^*,\rho)$. The following important result was established in 
 \cite{B-B-C}.\\

\noindent{\bf Result~4.}\label{RealTheo}
	Consider a function $S_0$ on some subset $\Omega_0$ of $\Omega$ with values in $B(\U, \Y)$. Then the following conditions are equivalent.
	\begin{enumerate}
		\item There exists an $S$ in $H^\infty_\Psi(\U, \Y)$ with $C_S =1$ such that $S|_{\Omega_0} =S_0$.
		\item $S_0$ has an {\em Agler decomposition} on $\Omega_0$, that is, there exists a completely positive kernel $\Gamma: \Omega_0 \times \Omega_0 \rightarrow B(\mathcal{C}_b(\Psi), B(\Y))$ so that
		\begin{align*}
		I_{\Y} - S_0(z) S_0(w)^* = \Gamma(z,w) (1- E(z) E(w) ^*) \ \ \text{for all} \ z, w \in \Omega_0.
		\end{align*}
		\item There exists a Hilbert space $\X$, a $*$-representation $\rho: \mathcal{C}_b(\Psi) \rightarrow B(\X)$ and a $\Psi$-unitary colligation $(V, \rho)$ such that writing $V$ as
		\[
		V =
		\bordermatrix{ & \X & \U \cr
			\X & A & B \cr
			\Y & C & D}, \qquad
		\]
		one has
		\begin{align}
		S_0(z) = D+ C(\I_{\X}- \rho(E(z)) A) ^{-1} \rho(E(z)) B \ \ \text{for all} \ z \in \Omega_0.
		\end{align}
		\item There exists a Hilbert space $\X$, a $*$-representation $\rho: \mathcal{C}_b(\Psi) \rightarrow B(\X)$ and a $\Psi$-unitary colligation $(W, \rho)$ such that writing $W$ as
		\[
		W =
		\bordermatrix{ & \X & \Y \cr
			\X & A_1 & B_1 \cr
			\U & C_1 & D_1}, \qquad
		\]
		one has
		\begin{align}\label{transfer}
		S_0(z)^* = D_1+ C_1 (\I_{\X}- \rho(E(z))^* A_1) ^{-1} \rho(E(z))^* B_1 \ \ \text{for all} \ z \in \Omega_0.
		\end{align}
	\end{enumerate}
	
The following lemma is proved in \cite{B-B-C} which was one of tools in establishing the Result~4 above. Since it is needed in the proof of our main theorem, we state it here without its proof.
	
\begin{lemma} \label{L:existcompletepos}
Let $J : \Omega \times \Omega\rightarrow B(\Y)$ be a self-adjoint function. Suppose 
\begin{align}\label{TheJ}
J\oslash K: (z,w)\mapsto J(z,w)\otimes K(z,w)
\end{align}
is a positive kernel for every $B(\Y)$-valued kernel $K$, then there is a completely positive kernel 
$\GammaDomainRange$ such that
$$J(z,w)=\Gamma(z,w)(1- E(z)E(w)^*)\,\, \text{for all}\,\, z,w\in \Omega.$$
\end{lemma}

We now end this section with another lemma that shows an application of Result~\ref{RealTheo} to the finite interpolation problem.
\begin{lemma}\label{L:solintp}
Let $\underline{w}=\{w_j:j\in \nat\}\subset\Omega$ be a sequence in $\Omega$ and let $\underline{x}=\{x_j:\in \nat\}$ be a sequence of complex numbers. Then 
there exists $f\in H^{\infty}_{\Psi}(\mathbb{C})$ with $||f||_{\Psi}\leq C_{\underline{w}}$ and $f(w_j)=x_j$ if and only if for every $n\in\nat$
the matrix $$\Big(\big(C_{\underline{w}}^2-x_i\,\overline{x_j}\big)k(w_i,w_j)\Big)_{i,j=1}^n$$ is positive semi-definite for every
$k\in\mathcal{K}_{\Psi}$.
\end{lemma}
\begin{proof} 
First assume that there exists $f\in H^{\infty}_{\Psi}(\mathbb{C})$ with $||f||_{\Psi}\leq C_{\underline{w}}$ and such that $f(w_j)=x_j$. Then it follows\,---\,from part $(2)$ of Result~4\,---\,that $\big(\big(C_{\underline{w}}^2-x_i\,\overline{x_j}\big)k(w_i,w_j)\big)\geq 0$ for all $k\in\mathcal{K}_{\Psi}$.
Conversely, assume that the functional $(i,j)\mapsto\big(C_{\underline{w}}^2-x_i\,\overline{x_j}\big)k(w_i,w_j)$ is positive semi-definite for all 
$k\in\mathcal{K}_{\Psi}$. Let us denote by $\Omega_{0}:=\{w_i:i\in\nat\}$ and $y_i={x_i}/{C_{\underline{w}}}$. Then the function 
$\mathcal{J}:\Omega_0\times\Omega_0\lrarw\mathbb{C}$ defined by $\mathcal{J}(w_i,w_j)=(1-y_i\bar{y}_j)$ satisfies the following:
\[
{\mathcal{J}(w_i,w_j)}^{*}=\mathcal{J}(w_j,w_i) \ \ \ \text{and $\mathcal{J}(w_i,w_j)\,k(w_i,w_j)\geq 0$ \ $\forall$ $k\in\mathcal{K}_{\Psi}$.}
\]
Now by Lemma~\ref{L:existcompletepos}, there exists a completely positive kernel
$\Gamma:\Omega_0\times\Omega_0\lrarw{\mathcal{C}_b(\Psi)}^{*}$ such that 
\[
\mathcal{J}(w_i,w_j)=(1-y_i\overline{y_j})=\Gamma(w_i,w_j)\big(1-E(w_i)E(w_j)^{*}\big),\ \ \ \forall \ i, j.
\]
Let $S_0:\Omega_0\lrarw\mathbb{C}$ be defined by $S_0(w_j):=y_j$. Then from the equivalence of part $(1)$ and $(2)$ in Result~4, we get that there exists $S\in H^{\infty}_{\Psi}(\mathbb{C})$ with $||S||_{\Psi}\leq 1$ and such that $S|_{\Omega_0} =S_0$. Observe
that $\phi=C_{\underline{w}}S$ that has the desired properties. 
\end{proof}

\section{An Important Proposition}\label{S:ImpProp}
In this section, we shall prove a crucial proposition, namely: Proposition~\ref{P:Equivalence} below.
The proposition is an 
analogue, in our setting, of a result in the case of an irreducible complete Nevanlinna-Pick kernel due to Agler-McCarthy; see \cite[Theorem~9.46]{A-M}. In fact, as the later result plays a crucial role in the proof of (S+C)$\implies$(IM) in Result~2, similarly the upcoming proposition is at the heart of the proof of Theorem~\ref{T:maint}.
But first we need a few lemmas.

\begin{lemma}
Let $\U,\,\Y,\,\HH$ be Hilbert spaces with bases $\{u_{\alpha}\},\,\{y_{\beta}\},\,\{h_{\delta}\}$, respectively. For $A\in B(\U,\,\Y),\,B\in B(\Y,\,\HH),\,
C\in B(\U,\,\HH)$, define $A^{t}\in B(\Y,\,\U),\,B^{t}\in B(\HH,\,\Y),\, C^t\in B(\HH,\,\U)$ by first setting
\begin{align*}
\langle A^ty_\beta,\,u_\alpha\rangle:=&\langle Au_{\alpha},\,y_{\beta}\rangle\\ 
\langle B^t h_\delta,\,y_\beta\rangle:=&\langle B y_{\beta},\,h_{\delta}\rangle\\
\langle C^t h_\delta,\,u_\alpha\rangle:=&\langle C u_{\alpha},\,h_{\delta}\rangle
\end{align*}
and then extending linearly on linear combinations of basis elements. Then $(BA)^{t}=A^t\,B^t$ with respect to these given bases.
\end{lemma}
\begin{proof}
Note that 
\[
\Big{\langle} A^t\sum d_{\alpha}y_{\alpha},\,\sum a_{\beta}u_{\beta}\Big{\rangle}:=\sum_{\alpha,\beta}d_{\alpha}\bar{a}_{\beta}\langle A u_{\beta},\,y_{\alpha}
\rangle=\Big{\langle} A\Big(\sum_{\beta}\bar{a}_{\beta}u_\beta\Big),\, \sum_{\alpha}d_{\alpha}y_{\alpha}\Big{\rangle}.
\]
From this it follows that $A^t$ is bounded linear operator, and so are $B^t$, $(BA)^{t}$. Now 
\[
\langle(BA)^t h_{\delta},\,u_{\alpha}\rangle=\langle (BA)u_{\alpha},\,h_{\delta}\rangle=\Big{\langle} B\sum_{\beta}\langle A u_{\alpha},\,
y_{\beta}\rangle y_{\beta},\,h_{\delta}\Big{\rangle}=\,\sum_{\beta}\langle Au_{\alpha},\,y_{\beta}\rangle\,\langle B y_{\beta},\,h_{\delta}\rangle.
\]
Also
\[
\langle A^t B^t h_{\delta},\,u_{\alpha}\rangle\,=\,\Big{\langle} A^t\sum_{beta}\langle B^t h_{\delta},\,y_{\beta}\rangle y_{\beta},\,u_{\alpha}
\Big{\rangle}=\sum_{\beta}\langle B^t h_{\delta},\,y_{\beta}\rangle\,\langle A^t y_\beta,\,u_\alpha\rangle=\sum_{\beta}\langle B y_{\beta},\,h_\delta
\rangle\,\langle A u_{\alpha},\,y_{\beta}\rangle.
\]
Thus $(BA)^t=A^t\,B^t$. 
\end{proof}

\begin{lemma}\label{L:Phi_t}
Let $\U,\,\Y$ be Hilbert spaces with bases $\{u_{\alpha}\}$ and $\{y_{\beta}\}$ respectively. Then for any $\Phi\in\SA_{\Psi}(\U,\,\Y)$, there exists 
$\Phi ^t\in\SA_{\Psi}(\U,\,\Y)$ such that 
\begin{equation}\label{E:Phi_t}
\langle \Phi^t(z) y_\beta,\,u_\alpha\rangle=\langle \Phi(z) u_{\alpha},\,y_{\beta}\rangle
\end{equation}
for all $z\in\Omega$ and $\alpha,\,\beta$. 
\end{lemma}
\begin{proof}
Since $\Phi\in\SA_{\Psi}(\U,\,\Y)$, there exist a Hilbert space $\HH$ (with basis $\{h_{\delta}\}$), a unital $*$-representation
$\rho:\mathcal{C}_b(\Psi)\lrarw B(\HH)$ and a unitary $V$, where if we write 
\[
V=
\bordermatrix{ & \HH & \U \cr
	\HH & A & B \cr
	\Y & C & D}, \qquad
\]
then 
\[
\Phi(z)\,=\,D+C\,\rho(E(z))\big(\I_{\HH}-A\rho(E(z))\big)^{-1}B.
\]
Consider now $A^t,\,B^t,\,C^t,\,D^t$ and $V^t$, $\rho(E(z))^t$ with respect to the bases above. Observe $V^t$ is a unitary operator such that 
\[
V^t=\bordermatrix{ & \HH & \Y \cr
	\HH & A^t & B^t \cr
	\U & C^t & D^t}, \qquad
	\]
It is easy to 
see that $\rho^t:\mathcal{C}_b(\Psi)\lrarw B(\HH)$ is a unital $*$-representation. Now consider 
\[
\Phi^t(z):=D^t+B^t\,\rho^t(E(z))\big(\I_{\HH}-A^t\rho^t(E(z))\big)^{-1}C^t.
\]
Since $V^t$ is unitary, we know that $\Phi^t\in\SA_{\Psi}(\Y,\,\U)$.
The identity \eqref{E:Phi_t} follows from this later observation. 
\end{proof}

\begin{lemma}\label{L:l^2valued}
Let $\{\phi_n\}\in\SA_{\Psi}(\mathbb{C})$ be a sequence of $\mathbb{C}$-valued functions. Then there exists a $\Phi\in\SA_{\Psi}(l^2,\,l^2)$ such that 
\[
\langle \Phi(z)e_i,\,e_j\rangle=\phi_i(z)\,\delta_{i,j} \ \ \ \forall i,j \ \ \ \text{and $z\in\Omega$}.
\]
Here, $\{e_i:i\in\nat\}$ is the standard orhthonormal basis of $l^2$. 
\end{lemma}
\begin{proof}
For each $n$, we consider $\phi_n$. Then by the realization theorem for each $\phi_n$, there exists a Hilbert space $\HH_n$, a unital $*$-representation
$\rho_n:\mathcal{C}_{b}(\Psi)\lrarw B(\HH_n)$ and a unitary operator 
\[
V_n =
\bordermatrix{ & \HH_n & \mathbb{C} \cr
	\HH_n & A_n & B_n \cr
	\mathbb{C} & C_n & D_n}, \qquad
\]
such that 
\[
  \phi_n(x) = D_n + C_n\,\rho_n (E(x))(\mathbb{I}_{\HH_n}- A_n\,\rho_n (E(x)))^{-1}B_n  \ \ \forall \ x \in \Omega.
   \]
   Consider now the operators
   \begin{enumerate}
   \item $D:l^2\lrarw l^2$ defined by $D:=\oplus_{n\in\nat}\,D_n$,
   \item $A:\oplus_{n\in\nat}\HH_n\lrarw \oplus_{n\in\nat}\HH_n$ defined by $A:=\oplus_{n\in\nat}\,A_n$,
   \item $B:l^2\lrarw\oplus_{n\in\nat} \HH_n$ defined by $B:=\oplus_{n\in\nat}\,B_n$ and,
   \item $C:\oplus_{n\in\nat} \HH_n\lrarw l^2$ defined by $C:=\oplus_{n\in\nat}\,C_n$.
   \end{enumerate}
   We also define $\rho:\mathcal{C}_{b}(\Psi)\lrarw B(\oplus \HH_n)$ by setting $\rho(\delta)=\oplus_{n\in\nat}\,\rho_n(\delta)$. Then $\rho$ is a unital
   $*$-representation. Now consider 
   \[
V =
\bordermatrix{ & \HH & l^2 \cr
	\HH & A & B \cr
	l^2 & C & D}, \qquad,
\]
where $\HH=\oplus_{n\in\nat}\,\HH_n$. Let us compute
\begin{align*}
V^*\,V={\begin{pmatrix}
  A^* & C^* \\
  B^* & D^*
 \end{pmatrix}}
 {\begin{pmatrix}
  A & B \\
  C & D
 \end{pmatrix}}
 =\bordermatrix{ & \HH & l^2 \cr
	\HH & A^*A+C^*C & A^*B+C^*D \cr
	l^2 & B^*A+D^*C & B^*B+D^*D}, \qquad.
 \end{align*}
   Since each of the operators $A,B,C,D$ are diagonal, while computing the products $A^*A,\,C^*C$ etc. the  
   corresponding diagonal blocks have to 
   be multiplied. From this and that each $V_n$ is unitary, it follows that $V$ is unitary.
   \smallskip
   
   Consider 
   \[
   \Phi_(z) = D+ C\,\rho(E(z))(\I_{\HH}- A\,\rho (E(z)))^{-1}B  \ \ \forall \ z \in \Omega.
   \]
   Then $\Phi\in\SA_{\Psi}(l^2,\,l^2)$ as $\Phi$ is the transfer function of a unitary colligation. Note that:
   \[
   \I_{\HH}-A\,\rho (E(z))=\oplus_{n\in\nat} (\I_{\HH_n}- A_n\,\rho_n (E(z)))
   \]
   and hence 
   \[
   \big[\I_{\HH}-A\,\rho (E(z))\big]^{-1}=\oplus_{n\in\nat} \big[\I_{\HH_n}- A_n\,\rho_n (E(z))\big]^{-1}.
   \]
   It follows from this that $\Phi(z)=\oplus_{n\in\nat}\,\phi_n(z)$ and thus $\Phi(z)\,e_i=\phi_i(z)\,e_i$ and consequently we have:
   $\langle\Phi(z)\,e_i,\,e_j\rangle=\phi_i(z)\langle e_i,\,e_j\rangle=\phi_i(z)\delta_{ij}$.
   \end{proof}
   
   \begin{proposition}\label{P:Equivalence}
   Let $\Omega$ be a bounded domain in $\mathbb{C}^n$ and let $\Psi$ be a family of test functions and consider the set 
   $\mathcal{K}_{\Psi}$ of $\Psi$-admissible $\mathbb{C}$-valued kernels on $\Omega$. Then
   \begin{enumerate}
   \item There exists $N>0$ such that the Grammian $G_k\geq (1/N)\,\mathbb{I}$ for all $k\in\mathcal{K}_{\Psi}$ if and only if there exists $\Phi\in
   H^{\infty}_{\Psi}(\mathbb{C},\,l^2)$ with $||\Phi||_{\Psi}\leq \sqrt{N}$. 
   \smallskip
   
   \item There exists an $M>0$ such that the Grammian $G_k\leq M\,\mathbb{I}$ for all $k\in\mathcal{K}_{\Psi}$ if and only there exists $\varphi\in
   H^{\infty}_{\Psi}(l^2,\,\mathbb{C})$ with $||\varphi||_{\Psi}\leq \sqrt{M}$. 
   \end{enumerate}
   \end{proposition}
   
   \begin{proof} Let us start with establishing the part $(1)$ above. 
	\begin{enumerate}
		\item Suppose there is an $N>0$ such that $G_k \geq \frac{1}{N}\cdot \mathbb{I}$ for all $k\in 
		\mathcal{K}_\Psi$. This is equivalent to 
		\begin{align*}
		(N-\delta_{ij})k(w_i,w_j)\geq 0
		\end{align*}
		for all $k\in \mathcal{K}_\Psi$. By Lemma~\ref{L:existcompletepos} there is a completely positive kernel
		$\Gamma:\Omega_0 \times \Omega_0\rightarrow \mathcal{C}_b (\Psi)^*,\,\,(\Omega_0=\{w_j:j\geq 1\})
		$ such that 
		\begin{align}\label{N_Del_ij}
		N- \delta_{ij}=\Gamma(w_i,w_j)(1-E(w_i)E(w_j)^*)
		\end{align}
		Now by \cite[Proposition~3.3]{D-M}, there is a Hilbert space $\mathscr{E}$, a function $L:
		\Omega_0\longrightarrow B(\mathcal{C}_b (\Psi),
		\mathscr{E})$ and a unital $*$-representation $\rho:\mathcal{C}_b (\Psi) \longrightarrow B(\mathscr{E})
		$ such that
	        \begin{align*}
		\Gamma(x,y)(f g^*)=\langle L(x) f, L(y)g\rangle\,\,\text{and}\,\, L(x)(fg)=\rho(f)L(x)(g)
		\end{align*}
	for all $f,g\in \mathcal{C}_b (\Psi), x,y \in \Omega_0$. So (\ref{N_Del_ij}) can be rewritten as 
	\begin{align*}
	\Bigg\langle\begin{pmatrix}
	\rho(E(w_i))L(w_i)(1)\\
	\sqrt{N}
	\end{pmatrix},\begin{pmatrix}
	\rho(E(w_j))L(w_j)(1)\\
	\sqrt{N}
	\end{pmatrix}\Bigg\rangle_{\mathscr{E}\oplus\mathbb{C}} =\Bigg\langle\begin{pmatrix}
	L(w_i)(1)\\
	e_i
	\end{pmatrix},\begin{pmatrix}
	L(w_j)(1)\\
	e_j
	\end{pmatrix}\Bigg\rangle_{\mathscr{E}\oplus\, l^2}
	\end{align*}
	for all $i,j$ where $\{e_j:j\geq 1\}$ is the standard basis for $l^2$.
	
	By a certain argument, it is easy to see that there is a unitary operator $V:\mathscr{E}\oplus\mathbb{C}\rightarrow \mathscr{E}\oplus\, l^2$ (adding an infinite dimensional Hilbert space to $\mathscr{E}$ if necessary) that sends $\begin{pmatrix}
	\rho(E(w_i))L(w_i)(1)\\
	\sqrt{N}
	\end{pmatrix}$ to $\begin{pmatrix}
	L(w_i)(1)\\
	e_i
	\end{pmatrix}$ for all $i\geq 1$. Let us write 
	\begin{align*}
	V =
	\bordermatrix{ & \E & \mathbb{C} \cr
		\E & A & B \cr
		l^2 & C & D},
	\end{align*}
	and take $F(z)=D+ C(\I_{\E}- \rho(E(z)) A) ^{-1} \rho(E(z)) B$. Then 
	$F\in \SA_\Psi (\mathbb{C},l^2)$. If we set $\Phi =\sqrt{N}F$, then $\Phi$ satisfies 
	\begin{align*}
	||\Phi||_\Psi \leq \sqrt{N}\,\,\text{and}\,\, \Phi(w_i)(1)=e_i
	\end{align*}
	for all $i\geq 1$.
	
	Conversely, let there be an element ${\phi}^\prime\in H^\infty_\Psi (\mathbb{C},l^2)$ such that 
	$||{\phi}^{\prime}||_\Psi \leq \sqrt{N}\,\,\text{and}\,\, {\phi}^{\prime}(w_i)(1)=e_i$ for all $i\geq 1$.
	 So $F^\prime ={\phi^\prime}/{\sqrt{N}}\in \SA_\Psi(\mathbb{C},l^2)$ and consequently
	 there is a Hilbert space
	  $\HH$, a unital $*$-representation $\mu:\mathcal{C}_b (\Psi)\rightarrow B(\HH)$ and a unitary 
	$$U=\bordermatrix{ & \HH & \mathbb{C} \cr
		\HH & A^\prime & B^\prime \cr
		l^2 & C^\prime & D^\prime}$$
	such that $F^\prime (z)=D^\prime+ C^\prime(\I_{\HH}- \mu(E(z)) A^\prime) ^{-1} \mu(E(z)) B^\prime$ foer all $z\in \Omega$. Set $h_i =(\I_{\HH}- \mu(E(w_i)) A^\prime) ^{-1} B^\prime \sqrt{N}$ for all $i$. Then we have that the unitary $U$ sends $\begin{pmatrix}
	\mu(E(w_i)) h_i\\\sqrt{N}
	\end{pmatrix}$ to $\begin{pmatrix}
	h_i\\e_i
	\end{pmatrix}$. Using these facts we obtain 
	\begin{align}\label{Converse_N_Del_ij}
	N-\delta_{ij}=\langle\mu(1-E(w_i)E(w_j)^*)h_i,h_j\rangle.
	\end{align}
	It is easy to see that the map $\Gamma^\prime: \Omega\times \Omega\rightarrow\mathcal{C}_b (\Psi)^*$
	 defined by 
	\begin{align}\label{Positive_kernel_associated_N_delta_ij}
	\Gamma^\prime (z,w)(f)=\langle\mu (f)h(z),h(w) \rangle
	\end{align} where $h(z)=(\I_{\HH}- \mu(E(z)) A^\prime) ^{-1} B^\prime \sqrt{N}$,
	is a completely positive kernel. Now define $\widetilde{\phi}:\Omega_0\rightarrow B(l^2,\mathbb{C})$ by $
	\widetilde{\phi}(w_i)^* =B_i ^*$ where $B_i(\sum l_i e_i)=l_i$, $l_i\in \mathbb{C}$. Then
	 (\ref{Converse_N_Del_ij}) takes the form 
	\begin{align*}
	1- \frac{\widetilde{\phi}(w_i)}{\sqrt{N}}\,\frac{\widetilde{\phi}(w_j)^*}{\sqrt{N}}=\frac{1}{N}\, \Gamma^\prime
	(1-E(w_i)E(w_j)^*)
	\end{align*}
	for all $i,j$. By Result~$4$, (the equivalence of part $(1)$ and $(2)$), ${\widetilde{\phi}}/{\sqrt{N}}$ can 
	be extended to an element of 
	$\SA_\Psi(l^2, \mathbb{C})$ and the definition of $\SA_\Psi(l^2, \mathbb{C})$ yields that 
	\begin{align*}
	(N - \delta_{ij})k(w_i,w_j)\geq 0
	\end{align*}
	for all admissible kernel $k$. This completes the proof of the first part.\\
	
	\item For $M>0$ and an admissible kernel $k\in \mathcal{K}_{\Psi}$, the condition $G_k \leq M\cdot 
	\mathbb{I}$ is equivalent to
	 $(M \delta_{ij} -1)\,k(w_i,w_j)\geq 0$.
	So when $G_k \leq M\cdot \mathbb{I}$ for every admissible kernel $k$, following the same procedure as in
	 part $1$ of we can show that there is a
	$\phi^\prime \in H^\infty _\Psi (l^2, \mathbb{C})$ such that 
	$$||\phi^\prime||_\Psi\leq \sqrt{M}\,\,\text{and}\,\,\phi^\prime(w_j)(e_j)=1.$$
	Conversely, let there be a $\phi\in H^\infty _\Psi (l^2, \mathbb{C})$ such that $$||\phi||_\Psi\leq \sqrt{M}\,\,\text{and}\,\,\phi(w_j)(e_j)=1.$$ 
	So $F={\phi}/{\sqrt{M}}\in \SA_\Psi (l^2,\mathbb{C})$ and consequently, we can find a Hilbert space $\HH$, a unital $*$-representation $\rho :\mathcal{C}_b (\Psi)\rightarrow B(\HH)$ and a unitary \begin{align*}
	V =
	\bordermatrix{ & \HH & l^2 \cr
		\HH & A & B \cr
		\mathbb{C} & C & D}
	\end{align*}
	such that $F(z)=D+C\rho(E(z))(\I_\HH -A \rho(E(z)))^{-1}B$ for all $z\in \Omega$. Let $h_j =(I_\HH -A \rho(E(w_i)))^{-1}B\sqrt{M}(e_i)$. Then the unitary operator $V$ sends $\begin{pmatrix}
	\rho(E(w_i))h_i\\
	\sqrt{M}e_i
	\end{pmatrix}$ to $\begin{pmatrix}
	h_i\\
	1
	\end{pmatrix}$. From this we deduce
	\begin{align}\label{M_Delta_ij}
	M \delta_{ij} -1 =\langle\rho (1-E(w_i)E(w_j)^*)h_i,h_j\rangle .
	\end{align}
	Define $\Gamma :\Omega_0\times \Omega_0\rightarrow\mathcal{C}_b (\Psi)^*$ by \begin{align}
	\Gamma(w_i,w_j)(\delta)=\langle\rho (\delta)h_i,h_j\rangle,\delta \in \mathcal{C}_b (\Psi).
	\end{align}
	
	From (\ref{Positive_kernel_associated_N_delta_ij}), we know that it is completely positive.
	\smallskip
	
	\noindent{\bf Claim.} $(M \delta_{ij} -1)k(w_i,w_j)=\Gamma(w_i,w_j)(1-E(w_i)E(w_j)^*)k(w_i,w_j)$ is positive
	 for every admissible kernel $k$.
	 \smallskip
	
	Fix a $k\in \mathcal{K}_\Psi$ and define $\Gamma _k :\Omega_0 \times \Omega_0 \rightarrow B(\mathcal{C}_b (\Psi))$ by 
	\begin{align*}
	\Gamma_k (w_i,w_j)(\delta)=(1-E(w_i)E(w_j)^*)k(w_i,w_j)\delta, \ \ \ \delta\in \mathcal{C}_b (\Psi).
	\end{align*}
	We claim that $\Gamma_k$ is completely positive. To see that, take $a_1,a_2,\ldots a_n, b_1,b_2,\ldots b_n \in\Cbpsi$ and consider the expression 
	\begin{align*}
	\sum_{i,j=1}^{n} b_i ^* \Gamma_k (w_i,w_j)(a_i ^* a_j)b_j
	=\sum_{i,j=1}^{n} (a_i b_i)^* (a_j b_j)(1-E(w_i)E(w_j)^*)k(w_i,w_j).
	\end{align*}
	It is an element of $\Cbpsi$. Evaluating it at any $\psi\in \Psi$ and using the fact that $k$ is admissible, we find that $\sum_{i,j=1}^{n} b_i ^* \Gamma_k (w_i,w_j)(a_i ^* a_j)b_j$ is a positive element of $\Cbpsi$. This proves the claim that $\Gamma_k$ is completely positive.
	
	Now for $k$ as above and any $i,j$, define $F_k (i,j)=(1-E(w_i)E(w_j)^*)k(w_i,w_j)$. Similar argument as above yields that the matrix $F_k =(F_k (i,j))_{1\leq i,j\leq n}$ is a positive matrix with entries in $\Cbpsi$ (see Lemma IV.$3.2$ in \cite{Takesaki}). Again by Lemma IV.$3.1.$ in \cite{Takesaki}, $F_k$ can be written as a sum of finitely many matrices of the form $C=(a_i ^* a_j)_{1\leq i,j\leq n}$, $a_i \in \Cbpsi$. That is, there is a positive integer $l$ such that for any $i$ and $j$
	\begin{align*}
	F_k (i,j)=\sum_{m=1}^{n} a_{m_i} ^* a_{m_j}
	\end{align*}
	with $a_{m_i}\in \Cbpsi$. Now if $\Gamma:\Omega_0\times \Omega_0 \rightarrow \Cbpsi^*$ is completely positive, then for $w_1,w_2,\ldots w_n \in \Omega_0$ and $\alpha_1,\alpha_2,\ldots \alpha _n\in \mathbb{C}$ we have 
	\begin{align*}
	\sum_{i,j=1} ^{n}\overline{\alpha_i}\alpha_j \Gamma(w_i,w_j)(F_k (i,j))
	=\sum_{m=1}^{l}\Big\{\sum_{i,j=1}^{n} \overline{\alpha_i}\alpha_j \Gamma(w_i,w_j)(a_{m_i} ^* a_{m_j})\Big\}
	\end{align*}
	Since $\Gamma$ is completely positive, we get that the last expression is non-negative. Hence $(M \delta_{ij} -1)k(w_i,w_j)=\Gamma(w_i,w_j)(1-E(w_i)E(w_j)^*)k(w_i,w_j)$ is positive and this completes the proof.
	\end{enumerate}
   \end{proof}

\section{Proof of Theorem~\ref{T:maint}}\label{S:proofmaint}

\begin{proof}
	To see how one $(1)$ implies $(4)$,
	start with an interpolating sequence for $H^\infty_{\Psi}(\mathbb{C})$ and then proceeding exactly 
	as in the first half of the proof of \cite[Theorem~9.19]{A-M}, one gets both the conditions $(2)$ and $(3)$.
	Now suppose $(2)$ and $(3)$ 
	holds together. Then for every kernel $k\in\mathcal{K}_{\Psi}$ and $a=(a_i)\in l^2$ we have
	\[
	\frac{1}{N}\sum_{i\in\nat}\,|a_i|^2\,\leq\,{||}\sum_{i\in\nat}\,a_ig_i\,||^2\,\leq M\sum_{i\in\nat}\,|a_i|^2,
	\]
	where $g_i={k_i}/{||k_i||}$. Given $\{c_j:j\in\nat\}\in l^{\infty}_1(\nat)$, we have:
	\[
	||\sum_{i=1}^N a_i\bar{c}_ig_i||^2\leq M\sum_{i=1}^N |a_i\bar{c}_i|^2\leq M\sum_{i=1}^N|a_i|^2\leq MN\,
	||\sum_{i=1}^N a_i g_i||^2
	\]
	whence $||\sum_{i=1}^N a_i\bar{c}_ig_i||\leq \sqrt{MN}\,||\sum_{i=1}^N a_i g_i||$. It follows from
	this that the map $R:\text{span}\{g_i:i\in\nat\}\longrightarrow \text{span}\{g_i:i\in\nat\}$ that maps $g_i$ to 
	$\bar{c}_ig_i$ is a bounded linear operator with norm $\leq\sqrt{MN}$ and hence it extends over
	$\overline{\text{span}\{g_i:i\in\nat\}}$ with norm $\leq\sqrt{MN}$. This implies that the expression 
	\[
	\sum_{i,j\in\nat}\big(MN-c_j\bar{c}_i\big)k(w_i,\,w_j)
	\]
	is positive semi-definite for every $k\in\mathcal{K}_{\Psi}$. By Lemma~\ref{L:solintp}, there exists a $\phi\in
	H^{\infty}_{\Psi}(\mathbb{C})$ with $||\phi||_{\Psi}\leq \sqrt{MN}$ such that $\phi(w_j)=c_j$ for all 
	$j\in\nat$. This proves $(4)$ implies $(1)$.\\
	
	 We shall now show that $(2)$ is equivalent to $(3)$. By Proposition~\ref{P:Equivalence},
	 we have to show the following statements are equivalent.
	 \begin{itemize}
	\item[(i)] There is a $\phi\in H_\Psi ^\infty (\mathbb{C},l^2)$ with $||\phi||_\Psi \leq \sqrt{N}$ and 
	$\phi (w_j)=e_j$ for all $j\geq 1$, where $\{e_j:j\geq 1\}$ is the standard basis for $l^2$.
	\smallskip
	
	\item[(ii)] $\{w_j:j\geq 1\}$ is strongly separated and there is a $\widetilde{\phi}\in H^\infty_\Psi$ such that
	$||\widetilde{\phi}||_\Psi \leq \sqrt{M}$ and $\widetilde{\phi}(w_j)(e_j)=1$ for all $j\geq 1$.
	\end{itemize}
	
	Suppose $(i)$ holds. So with respect to the standard basis of $l^2$ we can write 
	\begin{align*}
	\phi =\begin{pmatrix}
	\phi_1\\
	\phi_2\\
	\cdot \\
	\cdot \\
	\end{pmatrix}
	\end{align*}
	such that $\phi_i (w_j)=\delta_{ij}$.
	\smallskip
	
	\noindent{{\bf Claim.} Suppose 
	\begin{align*}
	f =\begin{pmatrix}
	f_1\\
	f_2\\
	\cdot \\
	\cdot \\
	\end{pmatrix}\in H^\infty_\Psi (\mathbb{C}, l^2)
	\end{align*}
	with $||f||_\Psi \leq D$, then each $f_j\in H^\infty_\Psi(\mathbb{C})$ with $||f_j||_\Psi \leq D$.}
	\smallskip
	
	\noindent To see this, note that 
	\begin{align*}
	f(z)f(w)^* =\Big(f_i(z)\overline{f_j(w)}\Big)
	\end{align*}
	is an infinite matrix and for every $B(l^2)$-valued admissible kernel $K$, $$(D\cdot I_{l^2} -f(z)f(w)^*)\otimes K(z,w)$$
	is a positive $B(l^2\otimes l^2)$-valued kernel.
	Let $k$ be a $\mathbb{C}$-valued admissible kernel. Then, by the claim in Lemma~\ref{L:prodschuragler}, $k\cdot I_{l^2}$ is a $B(l^2)$-valued admissible kernel. Take $K=k\cdot I_{l^2}$, $c_1,c_2\ldots ,c_n\in \mathbb{C}$, $z_1,z_2,\ldots, z_n\in \Omega$, $u_i=e_m$ for some $m\geq 1$ and $v_i=c_i e_1$. Then
	\begin{align*}
	&\sum_{i,j=1}^{n}\Big\langle(D^2\cdot I_{l^2}- f(z_i)f(z_j)^*)\otimes K(z_i,z_j)u_j\otimes v_j, u_i\otimes v_i\Big \rangle\\
	&=\sum_{i,j=1}^{n}c_j \overline{c_i}(D^2 -f_m (z_i)\overline{f_m (z_j)})k(z_j,z_i).
	\end{align*}
	Since $f\in H^\infty_\Psi (\mathbb{C}, l^2)$ with $||f||_\Psi \leq D$, the expression at the right-hand side 
	above is non-negative. Hence each $f_m \in H^\infty_\Psi(\mathbb{C})$ with $||f_m||_\Psi \leq D $.
	\smallskip
	
	So $\phi_i\in H^\infty_\Psi(\mathbb{C})$ and $||\phi_i||\leq \sqrt{N}$ and consequently, $\{w_j:j\geq 1\}$ is
	strongly separated.
	Also $\phi (w_j)^t (e_j)=\phi_j (w_j) =1$, where $\phi (w_j)^t$ is as in Lemma~\ref{L:Phi_t}.
	Thus there there is an map $\widetilde{\phi}=\phi^t\in H^\infty_\Psi(l^2,\mathbb{C})$ such that
	 $||\widetilde{\phi}||_\Psi \leq \sqrt{N}$ and 
	$\widetilde{\phi}(w_j)(e_j)=1$ for all $j\geq 1$. So ${(ii)}$ holds.\\
	
	Now let ${(ii)}$ hold. So there is a $\widetilde{\phi}\in H^\infty_\Psi(l^2,\mathbb{C})$ such that
	 $||\widetilde{\phi}||_\Psi\leq \sqrt{M}$ and 
	$\widetilde{\phi}(w_j)(e_j)=1$ for all $j\geq 1$. Again by Lemma~\ref{L:Phi_t}, 
	\begin{align*}
	\widetilde{\phi}^t =\begin{pmatrix}
	\widetilde{\phi}_1\\
	\widetilde{\phi}_2\\
	\cdot \\
	\cdot \\
	\end{pmatrix}\in H^\infty_\Psi (\mathbb{C},l^2)
	\end{align*}
	and $||\widetilde{\phi}^t||\leq \sqrt{M}$. Since $\{w_j:j\geq 1\}$ is strongly separated, there is an $L>0$ and a sequence $\{\phi_j:j\geq 1\}$ in $H^\infty_\Psi$ such that $\phi_j(w_j)=\delta_{ij}$ and $||\phi_j||_\Psi \leq L$ for all $i,j$. Consider 
	
	\begin{align*}
	\Phi_1 =\begin{pmatrix}
	\phi_1\widetilde{\phi}_1\\
	\phi_2\widetilde{\phi}_2\\
	\cdot\\
	\cdot
	\end{pmatrix} =diag (\phi_1,\phi_2\cdots)\cdot \begin{pmatrix}
	\widetilde{\phi}_1\\
	\widetilde{\phi}_2\\
	\cdot\\
	\cdot
	\end{pmatrix}.
	\end{align*}
	By Lemma~\ref{L:l^2valued}, $diag (\phi_1,\phi_2\cdots)\in H^\infty_\Psi(l^2,l^2)$ with $\Psi$-norm at most $L$.
	Also $||\widetilde{\phi}||_\Psi\leq \sqrt{M}$. Hence by Lemma~\ref{L:prodschuragler}, $\Phi_1\in H^\infty_\Psi(\mathbb{C},l^2)$ with 
	$||\Phi_1||_\Psi \leq L\sqrt{M}$. Clearly, $\Phi_1 ^t \in H^\infty_\Psi (l^2,\mathbb{C})$ with $\Psi$-norm atmost $L\sqrt{M}$ and $\Phi_1 ^t (w_j)=e_j$.
	This is \textit{(i)} and our proof is complete.
	\end{proof}
	
We also present the following sufficient condition for a sequence to be interpolating.
\begin{proposition}
	Let $\{w_j:j\in \mathbb{N}\}$ be a sequence of points in $\Omega$. Given $\psi\in \Psi$, define $z_{\psi,j}=\psi(w_j)$. If there is an $\epsilon>0$ (that depends on $\psi$) such that 
\begin{align}\label{E:suffIP}
\prod_{j\neq m} \intf{}{}{|z_{\psi,j}-z_m|}{|1-z_{\psi,j}\overline{z_m}|}\geq \epsilon,\,\,\text{for all}\,\,m\geq 1,
\end{align}
then $\{w_j:j\in \mathbb{N}\}$ is an interpolating sequence for $H^{\infty}_\Psi(\mathbb{C})$.
\end{proposition}

\begin{proof}
It is not very difficult to see that the condition \eqref{E:suffIP} above is equivalent to the condition that the sequence $\{z_{\psi,j}\}$ in $\mathbb{D}$ is strongly separated. Hence by Result~1 of Carleson, we know that
the sequence $\{z_{\psi,j}\}$ is interpolating for $H^{\infty}(\mathbb{D})$. Now given an arbitrary sequence 
$\{c_j:j\in\nat\}\in l^{\infty}(\nat)$, choose $f\in H^{\infty}(\mathbb{D})$ such that $f(z_{\psi,j})=c_j$ for all $j$.
\smallskip 

\noindent{\bf Claim.} $f\circ\psi\in H^{\infty}_\Psi(\mathbb{C})$.
\smallskip

\noindent{To see the claim above, assume without loss of generality that $||f||_{\infty}\leq 1$. Then there exists 
a positive kernel $\Gamma$ on $\D$ such that }
\[
1-f(z)\,\overline{f(w)}\,=\,(1-z\,\bar{w})\,\Gamma(z,\,w).
\]
Now let $k\in\mathcal{K}_{\Psi}$ be given then 
\begin{align*}
 \big(1-f(\psi(z))\overline{f(\psi(w))}\big)\,k(z,w)\,&=
  \,(1-\psi(z)\,\overline{\psi(w)})\,\Gamma(\psi(z),\,\psi(w))\,k(z,\,w)\\
  \,&=\,(1-\psi(z)\,\overline{\psi(w)})\,k(z,\,w)\,\Gamma(\psi(z),\,\psi(w)).
 \end{align*}
Note that $(1-\psi(z)\,\overline{\psi(w)})\,k(z,\,w)$ is positive from the definition that $k$ is a $\Psi$ admissible kernel. Of course $\Gamma(\psi(z),\,\psi(w))$ is positive and hence it follows from this that 
$\big(1-f(\psi(z))\overline{f(\psi(w))}\big)\,k(z,w)$ is positive for every $\Psi$-admissible kernel $k$ from which
the claim follows.
\smallskip

Now notice that $f\circ\psi(w_j)=c_j$ for all $j$ and since $\{c_j:j\in\nat\}$ is 
arbitrary, we are done.
\end{proof}

\section{Examples}\label{examples}
In this section, we present several cases where Theorem~\ref{T:maint} could be applied to give a 
characterization of interpolating sequences. A few of the cases that
appear below have already been addressed in the literature.

\begin{enumerate}

\item{\bf The case of Polydisc.} This is the case when $\Omega=\mathbb{D}^n,\ n\geq 1$. The class of test functions that we consider in this case is $\Psi\,=\,\{z_1,\dots,z_n\}$. The case $n=1$ has already been discussed 
after the statement of Theorem~\ref{T:maint}.
\smallskip

In the case $n=2$, due to Ando's inequality, we know that 
$H^{\infty}_{\Psi}(\mathbb{C)}\,=\,H^{\infty}(\mathbb{D}^2):=$ the set of all bounded holomorphic functions on the bidisc, 
with the norm $||\bcdot||_{\Psi}$ being 
sup-norm. Therefore, Theorem~\ref{T:maint} provides a characterization of interpolating sequences in the bidisc
for the Banach algebra of bounded holomorphic functions with the sup-norm. This case was already considered by Agler--McCarthy in \cite{A-M-2}. 
\smallskip

In general, when $n\geq 3$, the Banach algebra $H^{\infty}_{\Psi}(\mathbb{C})$ does not coincide with the algebra of bounded holomorphic functions on $\mathbb{D}^n$; see e.g. \cite{Var74}. One can still apply Theorem~\ref{T:maint} to give a characterization of interpolating sequences in the polydisc for the algebra
$H^{\infty}_{\Psi}(\mathbb{C})$ with $\Psi\,=\,\{z_1,\dots,z_n\}$. 
\smallskip

\item{\bf The case of multiply connected planar domains.} Let $\Omega$ be a bounded domain in the complex 
plane with boundary consisting of $m+1$ disjoint smooth Jordan curves $\partial_0,\partial_1,\dots,\partial_m$
where $\partial_0$ denotes the boundary of the unbounded component of the complement of $\Omega$. Then
there exists a collection of test functions $\Psi\,=\,\{\psi_{\bf x}\,:\,{\bf x}\in\mathbb{T}_{\Omega}\}$, indexed by the so-called 
$\Omega$-torus $\mathbb{T}_{\Omega}\,:=\,\partial_0\times\partial_1\times\dots\times\partial_m$
 (see \cite[Section~4.1]{B_H} and \cite{D-M07}), such that $H^{\infty}_{\Psi}(\mathbb{C})$ is equal to the set of all bounded
holomorphic functions on $\Omega$ and the norm $||\bcdot||_{\Psi}$ being equal to the sup-norm.
Using this class of test functions, Theorem~\ref{T:maint} can be applied to characterize interpolating sequences for bounded
holomorphic functions in $\Omega$. 
\smallskip

\item{\bf The case of constrained algebras.} Let us denote by $A(\mathbb{D})$ the Banach algebra of holomorphic functions on $\mathbb{D}$
that are continuous upto $\overline{\mathbb{D}}$. The algebra $A(\mathbb{D})$ is called {\em the disk algebra}.
Let $B$ be a finite Blaschke product of degree $N\geq 2$ and consider the algebra $\mathcal{A}_B\,:=\,\mathbb{C}+B(z)\,A(\mathbb{D})$.
Let us denote by $H^{\infty}_B$ the weak-closure 
of $\mathcal{A}_B$. In \cite{D-U}, a minimal class of test functions has been constructed for the algebras 
$H^{\infty}_B$. Therefore, one could use Theorem~\ref{T:maint} to characterize interpolating sequences in
$\mathbb{D}$ for the constrained algebras $H^{\infty}_B$ using this class of test functions. 
\smallskip

\item{\bf The case of symmetrized bidisc.} In this case $\Omega\,=\,\mathbb{G}_2$ where $\mathbb{G}_2$ is the symmetrized bidisc
defined by $\mathbb{G}_2\,:=\,\pi_2(\mathbb{D}^2)$ where $\pi_2:\mathbb{C}^2\lrarw
\mathbb{C}^2$ is the symmetrization map defined by $\pi_2(z_1,\,z_2)\,:=\,(z_1+z_2,\,z_1z_2)$. A point in 
$\mathbb{G}_2$ is also denoted by a pair $(s,\,p)\in\mathbb{C}^2$. It is a fact due to Agler--Young (see e.g. \cite{B-S}) that
a point $(s,\,p)\in\mathbb{C}^2$ belongs to $\mathbb{G}^2$ if and only if for every
$\alpha\in\overline{\mathbb{D}}$ we have
\[
 \intf{}{}{2\alpha p-s}{2-\alpha s} \,\in\, \mathbb{D}.
 \]
Because of this one could consider the family $\Psi\,:=\{\psi_{\alpha}\,:\,\psi_{\alpha}(s,\,p)=(2\alpha p-s)/(2-
\alpha s)\,:\,\alpha\in\overline{\mathbb{D}}\}$, as a family of test functions for $\mathbb{G}_2$. Then it is a fact 
(see \cite{B-S}) that $H^{\infty}_{\Psi}(\mathbb{C})\,=\,H^{\infty}(\mathbb{G}_2)$, the set of bounded holomorphic functions
on $\mathbb{G}_2$ and the norm $||\bcdot||_{\Psi}$ being equal to sup-norm. Hence one could apply Theorem~\ref{T:maint} to
characterize interpolating sequences for $H^{\infty}(\mathbb{G}_2)$. This case too has been dealt and is due to Bhattacharyya--Sau 
\cite{B-S-1}. 
\end{enumerate}

\vspace{0.1in} \noindent\textbf{Acknowledgement:}
The first named author's research is supported by the University Grants Commission Centre for Advanced Studies. 
The second named author's research is supported by a postdoctoral fellowship at Harish-Chandra Research
Institute, Prayagraj (Allahabad).

\end{document}

%% file: includes.tex
\usepackage{thmtools,thm-restate}
\usepackage[margin=0.8in]{geometry}

\usepackage{varioref}
\usepackage[usenames,svgnames,xcdraw,table]{xcolor}
\definecolor{DarkBlue}{rgb}{0.1,0.1,0.5}
\definecolor{DarkGreen}{rgb}{0.1,0.5,0.1}
\usepackage[backref=page]{hyperref}
\hypersetup{
     colorlinks   = true,
     linkcolor    = DarkBlue, 
     urlcolor     = DarkBlue, 
	 citecolor    = DarkGreen 
}
\renewcommand*{\backref}[1]{}
\renewcommand*{\backrefalt}[4]{%
    \ifcase #1 (Not cited.)%
    \or        (Cited on page~#2)%
    \else      (Cited on pages~#2)%
    \fi}
\usepackage[capitalise,noabbrev]{cleveref}

\usepackage{appendix}

\usepackage{algorithmic}
\usepackage{algorithm}

\usepackage{graphicx,epsfig,amsmath,latexsym,amssymb,verbatim}

\newcommand{\extra}[1]{}

\usepackage{amsmath,amssymb,amsfonts}
\usepackage{amsthm}

\newtheorem{theorem}{Theorem}

\newtheorem{definition}[theorem]{Definition}
\newtheorem{lemma}[theorem]{Lemma}

\newtheorem{proposition}[theorem]{Proposition}

\theoremstyle{remark}

\def\squareforqed{\hbox{\rlap{$\sqcap$}$\sqcup$}}
\def\qed{\ifmmode\squareforqed\else{\unskip\nobreak\hfil
\penalty50\hskip1em\null\nobreak\hfil\squareforqed
\parfillskip=0pt\finalhyphendemerits=0\endgraf}\fi}
\def\endenv{\ifmmode\;\else{\unskip\nobreak\hfil
\penalty50\hskip1em\null\nobreak\hfil\;
\parfillskip=0pt\finalhyphendemerits=0\endgraf}\fi}

\renewenvironment{proof}{\noindent \textbf{{Proof~} }}{\qed\medskip}
\newenvironment{proof+}[1]{\noindent \textbf{{Proof #1~} }}{\qed\medskip}

\mathchardef\ordinarycolon\mathcode`\:
\mathcode`\:=\string"8000
\def\vcentcolon{\mathrel{\mathop\ordinarycolon}}
\begingroup \catcode`\:=\active
  \lowercase{\endgroup
  \let :\vcentcolon
  }

\newcommand{\nc}{\newcommand}

\nc{\barA}{\overline{A}}
\nc{\barB}{\overline{B}}
\nc{\barC}{\overline{C}}
\nc{\barD}{\overline{D}}
\nc{\barR}{\overline{R}}
\nc{\barX}{\overline{X}}
\nc{\barY}{\overline{Y}}
\nc{\barU}{\overline{U}}

\nc{\ALG}{\textsc{Alg}}